\documentclass[11pt,leqno]{article}
\usepackage{amssymb}
\usepackage[reqno]{amsmath}
\usepackage{amsthm}
\usepackage{amsfonts}
\usepackage{comment}
\usepackage{graphicx}
\usepackage{xcolor}
\usepackage{mathtools}
\mathtoolsset{showonlyrefs}
\usepackage{csquotes}
\usepackage{hyperref}
\usepackage{enumitem}
\usepackage{setspace} 

\usepackage{color}
\definecolor{gr}{rgb}{0,0.50, 0.50}
\definecolor{mg}{rgb}{0.85,0,0.85}

\usepackage{pgfplots}
	\usetikzlibrary{pgfplots.fillbetween,arrows}
	\tikzset{dimens1/.style={->,>=stealth}}
	\definecolor{blue1}{HTML}{f4eb0a}
	\definecolor{blue1a}{HTML}{2bc4b0}
	
\newcommand\blfootnote[1]{%
  \begingroup
  \renewcommand\thefootnote{}\footnote{#1}%
  \addtocounter{footnote}{-1}%
  \endgroup
}

\newcommand{\bk}{\color{black}}

\def\R{\mathbb{R}}

\def\ds{\displaystyle}
 
\def\l{\ell} 
\def\e{\varepsilon}

\def\gae{\Gamma^\varepsilon}

\def\ep{\varepsilon}

\def\dive{\textnormal{div}}

\def\rw{\rightarrow}
\def\ru{\rightharpoonup}
\def\G{\Gamma}
\def\n{\nabla}
\def\g{\gamma}

\newtheorem{theorem}{Theorem}[section]
\newtheorem{corollary}[theorem]{Corollary}
\newtheorem{lemma}[theorem]{Lemma}
\newtheorem{proposition}[theorem]{Proposition}
 \theoremstyle{definition}

\newtheorem{remark}[theorem]{Remark}
\numberwithin{equation}{section}
	\title{ Homogenization of a nonlinear elliptic problem\\ with imperfect rough interface} 
	
	\author{Sara Monsurr\`{o}$^{*\dagger}$, Carmen Perugia$^{\ddagger,\dagger}$ and Federica Raimondi$^{*\dagger}$}
	
	\date{}

\begin{document}

	%	\numberwithin{equation}{section}
	\maketitle
	\vskip -8mm
	{\scriptsize{* Universit\`a degli Studi di Salerno, Dipartimento di Matematica, Via Giovanni Paolo II 132, 84084 Fisciano (SA), Italy.\hskip 2mm Email: smonsurro@unisa.it,$\;$ fraimondi@unisa.it\\
\indent $\ddagger$ Universit\`a degli Studi del Sannio, Dipartimento di Scienze e Tecnologie, Via dei Mulini 74, 82100 Benevento, Italy.\hskip 2mm Email: cperugia@unisannio.it
\\ $\dagger$ Member of the Gruppo Nazionale per l'Analisi Matematica, la Probabilit\`a e le loro
Applicazioni (GNAMPA) of the Istituto Nazionale di Alta Matematica (INdAM).}}

	\begin{abstract}

This paper deals with an elliptic problem with a nonlinear lower order term set in an open bounded cylinder of $\R^N$, $N\geq 2$, divided into two connected components by an imperfect rough interface. More precisely, we assume that at the interface the flux is continuous and proportional, via a nonlinear rule, to the jump of the solution.

According to the amplitude of the interface oscillations and the proportionality coefficient, we derive different homogenized problems.

\blfootnote{\vskip 0.3mm}
	\end{abstract}
	
	\section{Introduction }
	
	\label{intro}
%More precisely, we want to study an eigenvalue problem in a composite $Q$ with rough interface, introduced for the first time  in the homogenization framework by P. Donato and A. Piatniski in \cite{donpiat}.

This paper concerns a nonlinear elliptic problem with a lower order term defined in an open bounded cylinder $Q=\omega \times ]-\l,\l[$ of $\R^N$,  with $N\geq 2$,   $\l>0$ and $\omega$ bounded smooth  domain of $\R^{N-1}$. The cylinder is divided into two connected components $Q^+_\ep$ and $Q^-_\ep$ by a rough surface $\gae$, $\e$ being a positive parameter tending to zero. This common interface is represented by the graph of a quickly oscillating function $\e$-periodic in the first $N-1$ variables, having oscillations' amplitude of order $\e^k$, with $k> 0$. This means  that, as $\ep$ goes to zero, the interface turns into a flat surface. 

Given $f\in L^2(Q)$, we consider the following problem:
	\begin{equation} \label{prob}
	\left\{\begin{array}{llll}
	\ds -\dive(D^\e\nabla u_\e)+h_1( u_\e)=f& \textnormal{in} \; Q\setminus \gae,\\[2mm]
	%-\dive(D_2^\e\nabla v_\e)+h_2( v_\e)=f_2& \textnormal{in} \; Q\setminus \gae,\\[1mm]
	\ds (D^\e\nabla u_\e)^+ \cdot \nu_\e = (D^\e\nabla u_\e)^- \cdot \nu_\e & \textnormal{on} \; \gae,\\[2mm]
	%(D_2^\e\nabla v_\e)^+ \cdot \nu_\e = (D_2^\e\nabla v_\e)^- \cdot \nu_\e & \textnormal{on} \; \gae,\\[1mm]
	\ds (D^\e\nabla u_\e)^+\cdot \nu_\e = -\e^{\gamma} h_2([u_\e])& \textnormal{on} \; \gae,\\[2mm]
       % (D_2^\e\nabla v_\e)^+\cdot \nu_\e = \e^{\gamma_2} h(v_\e^+-v_\e^-)& \textnormal{on} \; \gae,\\[1mm]
	\ds u_\e = 0& \textnormal{on} \; \partial Q,
	\end{array}\right.
	\end{equation}			
where $\nu_\ep$ is the unit outward normal to $Q^+_\ep$ and $[\cdot]$ is the jump of the solution across $\gae$. 
We assume that the matrix $D^\e$ is $\ep$-periodic, uniformly elliptic and bounded, and the nonlinear functions $h_1$ and $h_2$ satisfy some natural regularity and growth conditions. For more details about the setting of the problem
 we refer the reader to Section \ref{secsetting}.
 
 We are interested in studying  the asymptotic behaviour, as $\ep$ tends to zero, of the solutions of problem \eqref{prob}. The values of the parameter $\gamma\in\R$ in the jump condition and the parameter $k$ in the amplitude of the oscillations play a crucial role and give rise to three different limit cases denoted by A), B) and C). In all of them, we obtain a homogenized  elliptic equation with a nonlinear lower order term.

In cases A) and B), the limit domain is a cylinder made up of two components $Q^+$ and $Q^-$ separated by a flat interface $\Gamma_0$. In particular, in case A), we obtain an effective imperfect transmission problem where the
flux is proportional, by means of a nonlinear function, to the jump of the solution via different constants whose values depend on $k$ and $\gamma$. In case B), one obtains a homogenized problem equivalent to two independent problems set in  $Q^+$ and $Q^-$, with homogeneous  Neumann boundary conditions on $\Gamma_0$.

\noindent Finally, in case C), we find a homogenized  Dirichlet elliptic problem in the whole cylinder $Q$.

In \cite{donpiat}, the authors study, in the same domain $Q$, the homogenization of the stationary heat diffusion with a discontinuity of the temperature field  through the interface modeled by a jump of the solution linearly proportional to the heat flux (see also \cite{AmAnT}, \cite{AmAnT2}, \cite{CJ}, \cite{FMPhomo}, \cite{FMP}, \cite{MP1}, \cite{Rai},  and the references therein for physical justifications of the jump condition). The homogenization of a related eigenvalue problem has been considered in \cite{AMR}. In \cite{DonJOn}, it is analyzed the  parabolic version of \cite{donpiat}. The homogenization for the elliptic problem of \cite{donpiat} with an additional singular lower order term is treated in \cite{DoGia}. Nonlinear lower order and interfacial terms, as in our case, appear in \cite{donngu} where, by using the unfolding operator, it is investigated the effective thermal conduction in a two component composite with inclusions.

 On the other hand, as evidenced in \cite{AmAnG}, nonlinear interface conditions also appear in the modeling of the electrical conduction in biological tissues, as a consequence of an interfacial resistance property of the cell membranes. Moreover, nonlinear reaction and jump terms, as in \eqref{prob}, occur very often when studying some metabolic processes arising in biological cells, in which biochemical species can diffuse and react either in the cytosol or in the various organelles which are present in the cytoplasm. These species can even react on the organelles' membrane which is not perfectly smooth but it contains some irregularities which play an important role mainly in problems related to enzymatic kinetic.  A precise mathematical model for such processes in the non-stationary case and in a two component composite with inclusions is analyzed in \cite{Gahn} and \cite{Gahn2}.
In \cite{Conca2003},  the authors study some nonlinear models involving adsorption and reactions
arising in porous media.  The homogenization of nonlinear problems modelling the chemical reactive flows through the exterior of a domain containing periodically distributed reactive solid grains is considered in \cite{Conca2004}. 
In order to model some reaction-diffusion processes in porous media, in \cite {CPT}, it has been analyzed the macroscopic behavior of the solution of a parabolic and coupled system of three partial differential equations depending on a small positive parameter $\varepsilon$ related to  the size and periodicity of the structure. For reaction-diffusion equations we can refer, for instance, to \cite{Gahn}$\div$ \cite{Pop1}, \cite{Muntean1} $\div$ \cite{Neu-Jag} and to the papers quoted therein.

Our work is organized as follows.

 In Section \ref{secsetting} we describe the two component  composite cylinder with rough interface, introduce the problem together with its functional setting and state the main result. 

Section \ref{seccomp} is devoted to some technical tools related to the geometry of the domain. At first we adapt to our context certain estimates contained in \cite{CeFrPiat}. Then, we  recall some compactness results of \cite{donpiat}.  Successively, we generalize and unify  some  convergence results proved in \cite{donpiat}. This argument allows us to develop the homogenization process without the introduction of the different extension operators needed in \cite{donpiat}. Finally, we show an useful convergence result related only to the geometry of the interface and to the parameter $\g$.

The existence and uniqueness of a weak solution of problem \eqref{prob} and some a priori estimates are proved in Section \ref{secexuniq}. 

At last, in Section  \ref{sechomo}, we carry out the homogenization process. 
	
		\section{Setting of the problem and main result}\label{secsetting}
%\subsection{Problem setup}
		
		Throughout the paper, we use the notation given in \cite{donpiat} to specify the structure of $Q$ and to define the functional spaces involved. 
		\\
Let $N\geq 2$ and 
		\begin{itemize}
		\item[$\bullet$] $Y=]0,1[^N$ be the volume reference cell,
		\item[$\bullet$] $Y'=]0,1[^{N-1}$ be the surface reference cell,
		\item[$\bullet$] $\mathcal{M}_{E}(v)$ be the average on $E$ of any function $v \in L^1(E)$,  
		\item[$\bullet$] $\chi\strut_{E}$ be the characteristic function of any open set $E \subseteq \R^N$.
		\end{itemize}

		Our domain is characterized by a rough interface described by means of the graph of a function $g$ satisfying the following assumption:
		\medskip
		
		\textbf{A$_g$}) \quad The function $g:Y'\rw \mathbb{R}$  is positive, periodic,  Lipschitz continuous.
		\vskip 2mm
		
		We set 
		\begin{equation}\label{gbar}
		\ds \bar{g}=\max g.
		\end{equation}
	
Given  a positive parameter  $\e$ converging to zero, $k>0$,  and set $x'=(x_1, ..., x_{N-1})$,  the graph
   $$    \gae=\left\{x \in Q\ |\ x_N=\e^k g\left(\frac{x'}{\e}\right)\right\}$$
    represents an oscillating (rough) interface which divides the set $Q$ in two subdomains.
    \\
 More precisely, we denote by
    $$    Q_\e^+=\left\{x \in Q\ |\ x_N>\e^k g\left(\frac{x'}{\e}\right)\right\}$$  the upper part of $Q$
    and by 
    $$    Q_\e^-=\left\{x \in Q\ |\ x_N<\e^k g\left(\frac{x'}{\e}\right)\right\}$$
 its lower part (see Figure 1).
 \vfill
 \begin{figure}[h]
    \begin{center}
\begin{tikzpicture}[scale=1.5]
	\draw[dimens1](-1,0) -- (5,0);
	\draw[dimens1](0,-1.5) -- (0,1.5);
	 \draw[name path=A,black, thick,domain=4:16] plot(0.25*(\x, {0.75*(\x-floor(\x)))});
	\draw[name path=B,white] (1,-1) -- (4,-1);
	\draw[name path=C,white] (1,1) -- (4,1);
	\tikzfillbetween[of=A and B]{blue1a};
	\tikzfillbetween[of=A and C]{blue1};
	\draw[black,thick,domain=4:16] plot(0.25*(\x, {0.75*(\x-floor(\x)))});
	\draw[dimens1](-1,0) -- (5,0);
	\draw[dash pattern=on 5pt off 1.7pt] (5,1) -- (-0.1,1) node[left] {\(\ell\)};
	%\draw[dash pattern=on 5pt off 1.7pt] (-0.1,0) -- (-0.1,0) node[left] {\(0\)};
	\draw[dash pattern=on 5pt off 1.7pt] (5,-1)-- (-0.1,-1)node[left] {\(-\ell\)};
	\draw[dash pattern=on 5pt off 1.7pt] (5,0.41) -- (-0.1,0.41) node[left] {\(\varepsilon^k\bar{g}\)};
  	\draw[decorate,decoration={brace,mirror}] (1,-0.05) -- (1.25,-0.05) ;
	\node at (-0.13,-0.2) {\(0\)};
		\node at (4,1.5) {\(Q\)};
  	\node at (1.13,-0.25) {\(\varepsilon\)};
   	\node at (1.5,-0.5) {\(Q_\varepsilon^-\)};
 	\node at (1.5,0.7) {\(Q_\varepsilon^+\)};
 	\node at (3.5,-0.2) {\(\Gamma^\varepsilon\)};
\end{tikzpicture}
\caption{\it The cylinder $Q$ with rough interface $\gae$ }\label{dom1}
\end{center}
\end{figure}
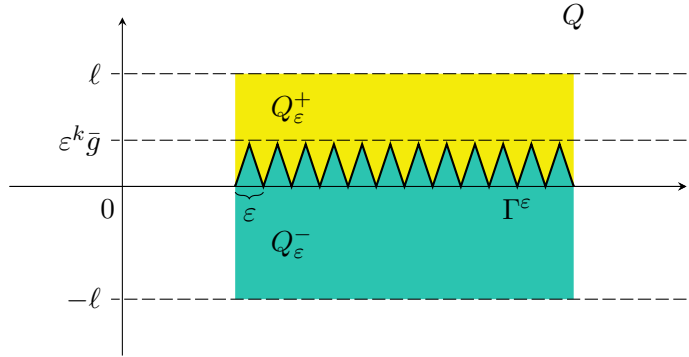 

As observed in \cite{donpiat}, in the case $0<k<1$, one has a fastly oscillating interface. When $k=1$, the interface $\gae$ is obtained by homothetic dilatation of the fixed function $y_N=g(y')$ in $\R^N$, hence, it presents a self-similar geometry. Finally, the flat case is obtained for $k>1$.
\\
In view of \eqref{gbar}, the oscillating surface is contained in the cylinder 
 \begin{equation}\label{se'}
\displaystyle S_\ep=\omega\times [0,\e^k\bar{g}] 
    \end{equation}
    and, as $\ep$ goes to zero,
     \begin{equation}\label{se}
\displaystyle |S_\ep|\to 0.
    \end{equation}
    
 \begin{remark} The cylindrical domain can be generalized  to an arbitrary smooth domain $Q$, where in any point of $\partial Q\cup  \{x | x_N = 0\}$ the normal to $\partial Q$ is not parallel to the $N$-th coordinate vector.
\end{remark}

    \vskip 2mm
Now, we state the specific hypotheses satisfied by  the matrix  $D^\e$ and the nonlinear lower order and  flux terms. Thus, let us consider the following assumptions:

\medskip

(\textbf{A$1$}) \quad For any $\e>0$, the coeffficients matrix $D^\e$ is defined by
$D^\e(x)=D\left(\frac{x}{\e}\right)$, where $D$ is $Y$-periodic, symmetric and satisfies
    \begin{equation*}
   \ds (D(x)\lambda,\lambda)\geq \alpha |\lambda|^2,\qquad  |D(x)\lambda|\leq \beta|\lambda|, \qquad \forall \lambda \in \R^N \hbox{ and a.e. in } Y,
    \end{equation*}
   with $\alpha, \beta\in \R$, $0<\alpha< \beta$.

\bigskip
   
(\textbf{A$2$})\quad The function $h_1:\mathbb{R}\rightarrow \mathbb{R}$ is such that
\begin{itemize}
\item [1.] $h_1$ is continuous;
\item [2.] $h_1$ is increasing and $h_1(0)=0$;
\item [3.] there exist a positive constant $C$ and an exponent $q_1$, with
\begin{equation}\label{assqi}
1\leq q_1<2\, \text{ if }N\in\{ 2,3\}\quad\text{and}\quad 1\leq q_1\leq \dfrac{N+2}{N-2}=\dfrac{2^*}{(2^*)'} \text{ if }N>3
\end{equation}
such that
\begin{equation}\label{assg*}
|h_1(z)|\leq C\left(1+| z|^{q_1}\right).
\end{equation}
\end{itemize}
  
 \bigskip
 
  (\textbf{A$3$})\quad The flux term $h_2:\mathbb{R}\rightarrow \mathbb{R}$ is such that
\begin{itemize}
\item [1.] $h_2$ is continuous; 
\item [2.] $h_2$ is an increasing function;
\item [3.] there exists a positive constant $C$ such that
\begin{equation}\label{assg**}
zh_2(z)\geq C z^2,\quad\forall s\in \mathbb{R};
\end{equation}
\item [4.] there exist a positive constant $C$ and an exponent $q_2$, with
\begin{equation}\label{asspi}
1\leq q_2< 2\, \text{ if }N\in\{2,3\}\quad\text{and}\quad 1\leq q_2\leq \dfrac{N}{N-2} \text{ if }N>3
\end{equation}
such that
\begin{equation}\label{assgi}
|h_2(z)|\leq C\left(1+| z|^{q_2}\right).
\end{equation}
  \end{itemize}
  
    \begin{remark}\label{remfi}
 Let us observe that assumption  (\textbf{A$2$}) implies
 $$z\,h_1 (z) \geq 0\quad \text{ for all } z\in \mathbb{R}.$$
  \end{remark}
  
\indent We now precise the functional framework. To this aim, for any function $v$ defined on $Q$, we set
$$v_\e^+=v_{|Q_\e^+}, \quad v_\e^-=v_{|Q_\e^-}.$$

Hence, we introduce the Hilbert space
$$W_0^\e=\{v \in L^2(Q)\ |\ v_\e^+ \in H^1(Q_\e^+), v_\e^- \in H^1(Q_\e^-), v=0 \hbox{ on }\partial Q\},$$ 
equipped with the norm
\begin{equation}\label{normwe0}
\|v\|_{W_0^\e}=\|\n v\|_{L^2(Q\setminus \gae)},
\end{equation}
where
\begin{equation}\label{grad}
\n v=\widetilde{\n v_\e^+}+\widetilde{\n v_\e^-}.
\end{equation}
In view of \eqref{grad}, we identify $\n v$ with the absolute continuous part of the gradient of $v$.

Moreover, to describe the limit domain, we set
$$Q^+=\{x\in Q\ |\ x_N>0\}, \quad Q^-=\{x\in Q\ |\ x_N<0\}, \quad \G_0=\{x\in Q\ |\ x_N=0\},$$
and
$$v^+=v_{|Q^+}, \quad v^-=v_{|Q^-},$$
for any function $v$ defined on $Q$ (see Figure \ref{dom2}).
%\vfill %\eject
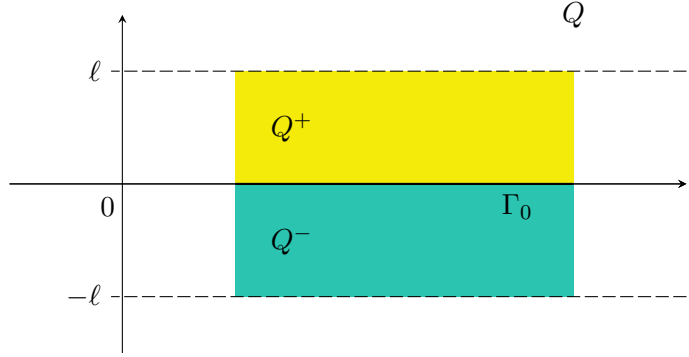
\begin{figure}[h]
\begin{center}
\begin{tikzpicture}[scale=1.5]
	%\scriptsize
	\draw[dimens1](-1,0) -- (5,0);
	\draw[dimens1](0,-1.5) -- (0,1.5);
	\draw[name path=A,white] (1,0) -- (4,0);
	\draw[name path=B,white] (1,-1) -- (4,-1);
	\draw[name path=C,white] (1,1) -- (4,1);
	\tikzfillbetween[of=A and B]{blue1a};
	\tikzfillbetween[of=A and C]{blue1};
	\draw[thick] (1,0) -- (4,0);
	\draw[dimens1](-1,0) -- (5,0);
	\draw[dash pattern=on 5pt off 1.7pt] (5,1) -- (-0.1,1) node[left] {\(\ell\)};
	\draw[dash pattern=on 5pt off 1.7pt] (5,-1)-- (-0.1,-1)node[left] {\(-\ell\)};
		\node at (4,1.5) {\(Q\)};
\node at (-0.13,-0.2) {\(0\)};
	\node at (1.5,-0.5) {\(Q^-\)};
 	\node at (1.5,0.5) {\(Q^+\)};
	 	\node at (3.5,-0.2) {\(\Gamma_0\)};
\end{tikzpicture}
\caption{\it The cylinder $Q$ with flat interface $\Gamma_0$} \label{dom2}
\end{center}
\end{figure}

Observe that
\begin{equation}\label{convchi}
\ds \chi_{Q_\ep^{\pm}}\rightarrow \chi_{Q^{\pm}}\quad \text{ strongly in }L^p(Q),\,1\leq p<+\infty \, \text{ and weakly* in }L^{\infty}(Q).
\end{equation}
Then, we define the Hilbert space
$$W_0^0=\{v \in L^2(Q)\ |\ v^+ \in H^1(Q^+), v^- \in H^1(Q^-), v=0 \hbox{ on }\partial Q\},$$ 
endowed with the norm
\begin{equation}\label{normw00}
\|v\|_{W_0^0}=\|\n v\|_{L^2(Q\setminus \G_0)},
\end{equation}
with $\n v$ defined similarly to \eqref{grad}.
Let us observe that \eqref{normwe0} and \eqref{normw00} are norms, as a consequence of the Poincar\'{e} inequality. In particular, we explictly observe that, for any $v \in W^\e_0$, there exists a positive constant C, independent of $\e$, such that
\begin{equation}\label{P}
\|v\|_{L^2(Q)}\leq  C \|\n v\|_{L^2(Q\setminus \gae)}.
\end{equation}
\indent The variational formulation of problem \eqref{prob} is then given by
\begin{equation} \label{probvar}
	\left\lbrace
	\begin{array}{l}
	\ds\text{Find } u_\e \in W_0^\e  \text{ such that  }  \\[5mm]
	{\ds \int_{Q \setminus \gae} D^\e \n u_\e \n \varphi\, dx }+ \e^{\gamma} {\displaystyle \int_{\gae}h_2(u^+_\e - u^-_\e)(\varphi_{\e}^+-\varphi_{\e}^-)\, d\sigma}
+{\ds \int_{Q} h_1(u_\e) \varphi\, dx}=
{\ds \int_{Q} f \varphi\, dx},\\[5mm]
\text{ for all } \varphi \in W_0^\e.
	\end{array}
	\right.
	\end{equation}	
 Under our assumptions the above integrals make sense and problem \eqref{probvar} admits a unique solution (see Section \ref{secexuniq} for details).

		\subsection{Statement of the main result}
In order to describe the asymptotic behaviour of the sequence of solutions of problem \eqref{probvar}, let us introduce the homogenized tensor $D^0$ (see also \cite{ben}) defined by
\begin{equation}\label{hommat}
\ds D^0\lambda=\mathcal{M}_{Y}(D(\n \omega_{\lambda})),
\end{equation}
with $\omega_{\lambda}\in H^1(Y)$  solution, for any $\lambda\in \mathbb{R}^N$ of
\begin{equation}\label{aux}
\left\{
\begin{array}{ll}
\ds -\hbox{div} \left( D\n \omega_{\lambda}\right)=0&\text{ in }Y,\\[2mm]
\ds \omega_{\lambda}-\lambda\cdot y & Y- \text{ periodic},\\[2mm]
	\mathcal{M}_Y(w_\lambda-\lambda \cdot y)=0.

\end{array}
\right.
\end{equation}

\vskip 3mm
 The main result of our paper is given in the following theorem whose proof is contained in Section \ref{sechomo}:
\begin{theorem}\label{theo:homores2}
Under assumptions $(\textbf{A}_g)$, $(\textbf{A}1)$, $(\textbf{A}2)$, $(\textbf{A}3)$  and given  $f\in L^2(Q)$, let $u_\e$ be the solution of problem \eqref{probvar}. Then, there exists a function $u \in W^0_0$ such that
\begin{equation}\label{convfin}
\left\{\begin{array}{lll}
\ds u_\e \rw u &\hbox{ strongly in }L^2(Q),\\[2mm]
		\ds \chi\strut_{Q^+_\e}\n u_{\e} \ru \chi\strut_{Q^+} \n u &\text{ weakly in }(L^2(Q))^N,\\[2mm]
		\ds \chi\strut_{Q^-_\e}\n u_{\e} \ru \chi\strut_{Q^-} \n u &\text{ weakly in }(L^2(Q))^N.
		\end{array}\right.
		\end{equation}

\vskip 0.3cm

	\noindent \textbf{\emph{} Case A):} \,  $(k\geq 1$ and $\g=0)$\, {or}\, $(0<k<1$ and $\g=1-k)$
	
	\vskip 0.3cm
The limit function $u 	\in W^0_0$ is the unique solution of problem
\begin{equation}\label{prolimi1}
\left\{\begin{array}{ll}
\ds -\dive(D^0\n u)+h_1(u)=f & \textnormal{ in }Q\setminus \Gamma_0,\\[2mm]
\ds (D^0\n u)^+\cdot \nu=(D^0\n u)^-\cdot \nu & \textnormal{ on }\Gamma_0,\\[2mm]
\ds (D^0\n u)^+\cdot \nu=- {G}\,h_2(u^+-u^-) & \textnormal{ on }\Gamma_0,\\[2mm]
\ds u=0 & \textnormal{ on }\partial Q,
\end{array}
\right.
\end{equation}
where $\nu$ is the unit outward normal to $Q^+$ and
\begin{equation}\label{H}
\ds {G}=\left\{
\begin{array}{ll}
\ds \mathcal{M}_{Y'}\left(\sqrt{1+\left|\n_{y'} g(y')\right|^2}\right) &\hbox{ if } k=1, \g=0,\\[2mm]
\ds 1 &\hbox{ if } k>1, \g=0,\\[2mm]
\ds \mathcal{M}_{Y'}\left(\left|\n_{y'} g(y')\right|\right) &\hbox{ if } 0<k<1, \g=1-k.
\end{array}	
\right.
\end{equation}
\vskip 5mm
	\noindent \textbf{\emph{} Case B):} \,  $(k\geq 1$ and $\g>0)$ \, {or} \, $(0<k<1$ and $\g>1-k)$
	\vskip0.3cm
The limit function $u	\in W^0_0$
 is the unique solution of problem
\begin{equation}\label{prolimi2}
\left\{\begin{array}{ll}
\ds -\dive(D^0\n u)+h_1(u)=f & \textnormal{ in }Q\setminus \Gamma_0,\\[2mm]
\ds (D^0\n u)^-\cdot \nu=(D^0\n u)^+\cdot \nu=0 & \textnormal{ on }\Gamma_0,\\[2mm]
\ds u=0 & \textnormal{ on }\partial Q,
\end{array}
\right.
\end{equation} where $\nu$ is the unit outward normal to $Q^+$. 

\vskip 5mm

\noindent  \textbf{\emph{} Case C):} \, $({k\geq 1}$ and ${\g<0})$ \, {or} \,   $({0<k<1}$ and ${\g<1-k})$

\vskip 0.3cm
The limit function $u$ belongs to $H^1_0(Q)$ and it is the unique solution of problem

\begin{equation}\label{prolimi3}
\left\{\begin{array}{ll}
\ds -\dive(D^0\n u)+h_1(u)=f & \textnormal{ in }Q,\\[2mm]
\ds u=0 & \textnormal{ on }\partial Q.
\end{array}
\right.
\end{equation}

\vskip 5mm
\end{theorem}
\begin{remark} 
We observe that  in the three limit problems we always get a homogenized  elliptic equation with a nonlinear lower order term. The main differences are due to the interfacial contribution. In particular, in case A), we obtain an effective imperfect transmission problem where the
flux is proportional, by means of a nonlinear function, to the jump of the solution via different constants whose values depend on $k$ and $\gamma$. In case B),  the flat surface $\Gamma_0$ is an isolating interface. Indeed, one obtains a homogenized problem equivalent to two independent problems set in  $Q^+$ and $Q^-$, with homogeneous Neumann boundary conditions on $\Gamma_0$ (see also \cite{donpiat}). Finally, in case C), the presence of the interface is neglectful since we find a homogenized  Dirichlet elliptic problem in the whole cylinder $Q$.  
\end{remark}

\section{Some preliminary and compactness results}\label{seccomp}
This section is devoted to some tools, related to the geometry of the domain, needed to study the asymptotic behaviour of problem \eqref{probvar}. 

We start stating a generalized version of a lemma proved in \cite{CeFrPiat}  (see also \cite{donpiat}). The proof can be easily obtained following along the lines the one of Lemma 1 in \cite{CeFrPiat}, with opportune modifications. 
\begin{lemma}\label{lemCeFrPiat} Let $p\geq1$. For any $w_1\in W^{1,p}(Q^-_\ep)$, $w_2\in W^{1,p}(Q^+_\ep)$ and $w_3\in W^{1,p}(S_\ep)$ it holds
\begin{equation}\label{CeFrPiat1}
\ds \left\|w_1\left(x',\ep^kg\left(\frac{x'}{\ep}\right)\right)-w_1\left(x',0\right)\right\|_{L^2(\omega)}\leq C{\ep^{k/p}}\|w_1\|_{W^{1,p}(Q^-_\ep)},
\end{equation} 
\begin{equation}\label{CeFrPiat2}
\ds \left\|w_2\left(x',\ep^kg\left(\frac{x'}{\ep}\right)\right)-w_2\left(x',\ep^k\bar{g}\right)\right\|_{L^2(\omega)}\leq C{\ep^{k/p}}\|w_2\|_{W^{1,p}(Q^+_\ep)},
\end{equation}
and
\begin{equation}\label{CeFrPiat3}
\ds \left\|w_3\left(x',\ep^k\bar{g}\right)-w_3\left(x',0\right)\right\|_{L^2(\omega)}\leq C{\ep^{k/p}}\|w_3\|_{W^{1,p}(S_\ep)},
\end{equation}
with $C$ positive constant independent of $\ep$.
\end{lemma}
Now, we recall a compactness result given in \cite{donpiat}. 
\begin{proposition}\label{propdonpiat1}
Let $\{w_\e \}$ be a family of functions in $W_0^\e$ such that
\begin{equation}\label{stima}
\ds \|w_\e\|_{W_0^\e}\leq C,
\end{equation}
with $C$ positive constant independent of $\e$. Then, there exist a subsequence (still denoted $\e$) and a function $w$
in $W_0^0$ such that
\begin{equation}\label{convw}
		\left\{\begin{array}{lll}
		\ds w_{\e}\rw w &\text{ strongly in } L^2(Q) \text{ and a.e. in }Q,\\[2mm]
		\ds \chi\strut_{Q^+_\e}\n w_{\e} \ru \chi\strut_{Q^+} \n w &\text{ weakly in }(L^2(Q))^N,\\[2mm]
		\ds \chi\strut_{Q^-_\e}\n w_{\e} \ru \chi\strut_{Q^-} \n w &\text{ weakly in }(L^2(Q))^N.
		\end{array}\right.
		\end{equation}
\end{proposition}

 In \cite{donpiat} the authors prove some convergences for the sequence of solutions $u_\ep$ of problem \eqref{probvar}  by using different extension operators according to the values of $k$.

In the next proposition, inspired by the argument used in \cite{donpiat} for $0<k<1$, we give a unified proof of these convergences valid for $k>0$ and for any uniformly bounded sequence $w_\ep$ in $W_0^\e$. This consents to develop the homogenization process without the introduction of the extension operators used in \cite{donpiat} for the case $k\geq1$, and allow us to treat all the cases  in a unique way.
\begin{proposition}\label{propdonpiat2}
Under the hypotheses of Proposition \ref{propdonpiat1},  there exists a subsequence (still denoted $\e$) such that
\begin{equation}%\label{convtrace'}
 \ds w^{\pm}_\e\left(\cdot,\ep^k g\left(\frac{\cdot}{\ep}\right)\right) \, \rw\,
w^{\pm}(\cdot,0) \quad \text{ strongly in } L^2(\omega).
\end{equation}
\end{proposition}
\begin{proof}
 To obtain the desired convergences we show that
\begin{equation}\label{k<1lim4}
\ds \left\|w^+_\ep\left(\cdot,\ep^kg\left(\frac{\cdot}{\ep}\right)\right)-w^+\left(\cdot,0\right)\right\|_{L^2(\omega)} \to 0,
\end{equation}
and
\begin{equation}\label{k<1lim5}
\ds \left\|w^-_\ep\left(\cdot,\ep^kg\left(\frac{\cdot}{\ep}\right)\right)-w^-\left(\cdot,0\right)\right\|_{L^2(\omega)} \to 0,
\end{equation} up to a subsequence.

We proceed into three steps.
In the first two steps we show \eqref{k<1lim4}, while \eqref{k<1lim5} is proved in the last one.
\smallskip

{\bf Step 1.} Let us prove that there exists a suitable extension operator denoted by $\check{w}^+_\ep$ such that
\begin{equation}\label{exconv1}
\ds \check{w}^+_\ep \to w^+\quad \text{ strongly in }L^2(Q^+).
\end{equation}  To this aim, we restrict the function $w^+_\ep$ to the set $Q^+\setminus S_\ep$ and define its extension $\check{w}^+_\ep$ to  $Q^+$ as follows: \bk
\begin{equation}\label{exdef}
\ds \check{w}^+_\ep=
\left\{\begin{array}{ll}
\ds w^+_\ep & \text{ in }Q^+\setminus S_\ep,\\[2mm]
\ds Pw^+_\ep & \text{ in } S_\ep,
\end{array}
\right.
\end{equation}
where $P$ denotes the usual extension operator in Sobolev spaces such that
\begin{equation}\label{exstima1}
\ds \|\check{w}^+_\ep\|_{H^1(Q^+)}\leq C\|w^+_\ep\|_{H^1(Q^+\setminus S_\ep)},
\end{equation}
with $C$ positive constant independent of $\ep$.
Due to the Sobolev embeddings and \eqref{stima}, inequality \eqref{exstima1} implies
\begin{equation}\label{exstima2}
\ds \|\check{w}^+_\ep\|_{L^{2^*}(Q^+)}\leq C,
\end{equation}
with $C$ positive constant independent of $\ep$. 

One has
\begin{equation}\label{exint1}
\begin{array}{c}
\ds \int_{Q^+}\left|\check{w}^+_\ep- w^+\right|^2\,dx=\int_{Q^+}\left|\chi_{Q^+\setminus S_\ep}\check{w}^+_\ep+\chi_{S_\ep}\check{w}^+_\ep- w^+\right|^2\,dx
\\[5mm]
\displaystyle \leq C\left(\int_{Q^+}\left|\chi_{Q^+\setminus S_\ep}\check{w}^+_\ep- w^+\right|^2\,dx+\int_{Q^+}\left|\chi_{S_\ep}\check{w}^+_\ep\right|^2\,dx\right)
\\[5mm]
\ds \leq C\left(\int_{Q^+}\left|w^+_\ep- w^+\right|^2\,dx+\int_{Q^+}\left|\chi_{S_\ep}\check{w}^+_\ep\right|^2\,dx\right).
\end{array}
\end{equation}
In view of \eqref{convw}, the first integral on the right hand side of \eqref{exint1} converges to zero. Concerning the second integral, by H\"older inequality, we get
\begin{equation*}
\ds \int_{Q^+}\left|\chi_{S_\ep}\check{w}^+_\ep\right|^2\,dx\leq \left|S_\ep\right|^{\frac{1}{(2^*/2)'}}\left\|\check{w}^+_\ep\right\|^2_{L^{2^*}(Q^+)}.
\end{equation*}
Thus, this integral goes to zero too, in view of \eqref{se} and \eqref{exstima2}. This concludes the proof of \eqref{exconv1}. 

\smallskip

{\bf Step 2.}
Here we prove \eqref{k<1lim4}. 
\\
Let us observe that
\begin{equation}\label{exnorm1}
\begin{array}{c}
\ds \left\|w^+_\ep\left(\cdot,\ep^kg\left(\frac{\cdot}{\ep}\right)\right)-w^+\left(\cdot,0\right)\right\|_{L^2(\omega)}\\[5mm]
\ds\leq \left\|w^+_\ep\left(\cdot,\ep^kg\left(\frac{\cdot}{\ep}\right)\right)-w_\ep^+\left(\cdot,\ep^k\bar{g}\right)\right\|_{L^2(\omega)}+\displaystyle\left\|w_\ep^+\left(\cdot,\ep^k\bar{g}\right)-w^+\left(\cdot,0\right)\right\|_{L^2(\omega)}.
\end{array}
\end{equation}
By \eqref{CeFrPiat2}, the first term on the right hand side of \eqref{exnorm1} tends to zero. 
Let us prove that the second term converges to zero as well. One has
\begin{equation}\label{exnorm2}
\begin{array}{c}
\displaystyle\left\|w_\ep^+\left(\cdot,\ep^k\bar{g}\right)-w^+\left(\cdot,0\right)\right\|_{L^2(\omega)}\\[5mm]
\displaystyle\leq \left\|w_\ep^+\left(\cdot,\ep^k\bar{g}\right)-w^+\left(\cdot,\ep^k\bar{g}\right)\right\|_{L^2(\omega)}+\left\|w^+\left(\cdot,\ep^k\bar{g}\right)-w^+\left(\cdot,0\right)\right\|_{L^2(\omega)}.
\end{array}
\end{equation}
Let us first show that
\begin{equation}\label{exnorm3}
\begin{array}{c}
 \ds \left\|w_\ep^+\left(\cdot,\ep^k\bar{g}\right)-w^+\left(\cdot,\ep^k\bar{g}\right)\right\|_{L^2(\omega)}= \left\|\check{w}_\ep^+\left(\cdot,\ep^k\bar{g}\right)-w^+\left(\cdot,\ep^k\bar{g}\right)\right\|_{L^2(\omega)}\\[5mm]
\ds \leq \left\|\check{w}_\ep^+\left(\cdot,0\right)-w^+\left(\cdot,0\right)\right\|_{L^2(\omega)}+\left\|\check{w}_\ep^+\left(\cdot,\ep^k\bar{g}\right)-w^+\left(\cdot,\ep^k\bar{g}\right)-\check{w}_\ep^+\left(\cdot,0\right)+w^+\left(\cdot,0\right)\right\|_{L^2(\omega)}\to 0.
\end{array}
\end{equation}
Indeed, by  \eqref{stima}, \eqref{exconv1} and  \eqref{exstima1} the difference  $(w^+-\check{w}^+_\ep)$ converges to zero weakly in $H^1(Q^+)$. Thus, by the compactness of the trace it holds
\begin{equation}\label{exnorm31}
\ds \left\|\check{w}_\ep^+\left(\cdot,0\right)-w^+\left(\cdot,0\right)\right\|_{L^2(\omega)}\to 0,
\end{equation}
 up to a subsequence.

\noindent Furthermore, by \eqref{CeFrPiat3} applied to the function $(\check{w}^+_\ep-w^+)$, one also has
\begin{equation}\label{exnorm32}
\ds \left\|\check{w}_\ep^+\left(\cdot,\ep^k\bar{g}\right)-w^+\left(\cdot,\ep^k\bar{g}\right)-\check{w}_\ep^+\left(\cdot,0\right)+w^+\left(\cdot,0\right)\right\|_{L^2(\omega)}\to 0.
\end{equation}
Combining \eqref{exnorm31} and \eqref{exnorm32} we get \eqref{exnorm3}.

\noindent Moreover, by the continuity of trace arguments, as observed in \cite{donpiat}, we get
\begin{equation}\label{exnorm2*}
\ds \left\|w^+\left(\cdot,\ep^k\bar{g}\right)-w^+\left(\cdot,0\right)\right\|_{L^2(\omega)}\to 0.
\end{equation}

\noindent Putting together \eqref{exnorm2}, \eqref{exnorm3} and \eqref{exnorm2*} we obtain
\begin{equation}\label{exnorm2'}
\begin{array}{c}
\displaystyle\left\|w_\ep^+\left(\cdot,\ep^k\bar{g}\right)-w^+\left(\cdot,0\right)\right\|_{L^2(\omega)}\to 0.
\end{array}
\end{equation}
\noindent The previous arguments, together with \eqref{exnorm1} and \eqref{exnorm2'}, entail the  claimed convergence \eqref{k<1lim4}.

\smallskip

{\bf Step 3.} The proof of \eqref{k<1lim5} is easier since there is no need to use extension operators. Indeed,
\begin{equation}\label{exnorm4}
\begin{array}{c}
\displaystyle\left\|w^-_\ep\left(\cdot,\ep^kg\left(\frac{\cdot}{\ep}\right)\right)-w^-\left(\cdot,0\right)\right\|_{L^2(\omega)}\\[5mm]
\ds\leq \left\|w^-_\ep\left(\cdot,\ep^kg\left(\frac{\cdot}{\ep}\right)\right)-w_\ep^-\left(\cdot,0\right)\right\|_{L^2(\omega)}+\displaystyle\left\|w_\ep^-\left(\cdot,0\right)-w^-\left(\cdot,0\right)\right\|_{L^2(\omega)}.
\end{array}
\end{equation}
By \eqref{CeFrPiat1} it holds
\begin{equation}\label{exconv3}
\ds \left\|w^-_\ep\left(\cdot,\ep^kg\left(\frac{\cdot}{\ep}\right)\right)-w_\ep^-\left(\cdot,0\right)\right\|_{L^2(\omega)}\to 0.
\end{equation}
On the other hand, by \eqref{P}, \eqref{stima} and the compactness of the trace operator, we get
\begin{equation}\label{exconv2}
\ds \left\|w_\ep^-\left(\cdot,0\right)-w^-\left(\cdot,0\right)\right\|_{L^2(\omega)}\to 0,
\end{equation}
up to a subsequence.\\
From \eqref{exnorm4}, \eqref{exconv3} and \eqref{exconv2}, one obtains \eqref{k<1lim5}.

This concludes the proof of the claimed convergence results.
\end{proof}

\begin{remark}[\cite{donpiat}]\label{reminterface}
We observe that the interface integral in \eqref{probvar}, in the coordinates $x'$,  reads as follows
	$$
	\begin{array}{c}
	{\displaystyle \int_{\gae}h_2(u^+_\e - u^-_\e)(\varphi_{\e}^+-\varphi_{\e}^-)\, d\sigma}=
	\\[5mm]
{\displaystyle \int_{\omega}h_2\left(u^+_\e\left(x',\ep g\left(\frac{x'}{\ep}\right)\right) - u^-_\e\left(x',\ep g\left(\frac{x'}{\ep}\right)\right)\right)}\times\\[5mm]
 \displaystyle{\left(\varphi_{\e}^+\left(x',\ep g\left(\frac{x'}{\ep}\right)\right)-\varphi_{\e}^-\left(x',\ep g\left(\frac{x'}{\ep}\right)\right)\right)}\times {\displaystyle\sqrt{1+\e^{2(k-1)}(\left|\n_{y'} g(y')\right|^2)_{|y'=x'/ \ep}}\, dx'}.
	\end{array}
	$$
	\end{remark}
Due to Remark \ref{reminterface}, to study the limit behaviour of problem \eqref{probvar}, we need the following convergence result related only to the geometry of the interface and to the parameter $\g$.
\begin{proposition}\label{propg}
Under the assumption (\textbf{A}$_g$) it holds
\begin{equation*}
\ds \ep^{\gamma}\sqrt{1+\e^{2(k-1)}(\left|\n_{y'} g(y')\right|^2)_{|y'=x'/ \ep}}\, \to\, G\quad \text{ weakly * in }L^\infty(Y'), 
\end{equation*}
where
\begin{equation}\label{Gint}
\ds {G}=\left\{
\begin{array}{ll}
\ds\mathcal{M}_{Y'}\left(\sqrt{1+\left|\n_{y'} g(y')\right|^2}\right) &\hbox{ if } k=1 \hbox{ and } \g=0,\\[2mm]
\ds1 &\hbox{ if } k>1 \hbox{ and }  \g=0,\\[2mm]
\ds\mathcal{M}_{Y'}\left(\left|\n_{y'} g(y')\right|\right) &\hbox{ if } 0<k<1  \hbox{ and }  \g=1-k,
\end{array}	
\right.
\end{equation}
while for  (\,$k\geq 1  \hbox{ and }  \g>0$)  or (\,$0<k<1 \hbox{and }\g>1-k$),
\begin{equation}\label{0int}
\ds \ep^{\gamma}\sqrt{1+\e^{2(k-1)}(\left|\n_{y'} g(y')\right|^2)_{|y'=x'/ \ep}}\, \to\, 0\quad \text{ weakly * in }L^\infty(Y').
\end{equation}
\end{proposition}
\begin{proof}
For $k=1$, one has
\begin{equation*}
\ds\sqrt{1+(\left|\n_{y'} g(y')\right|^2)_{|y'=x'/ \ep}}\, \to\, \ds\mathcal{M}_{Y'}\left(\sqrt{1+\left|\n_{y'} g(y')\right|^2}\right)\quad\text{ weakly * in }L^\infty(Y').
\end{equation*}
Therefore, for $\g=0$ we get the first condition in \eqref{Gint}, while for $\g>0$ we obtain \eqref{0int}.
\\
For $k>1$
\begin{equation*}
\ds \sqrt{1+\e^{2(k-1)}(\left|\n_{y'} g(y')\right|^2)_{|y'=x'/ \ep}}\, \to\, 1 \quad\text{ weakly * in }L^\infty(Y').
\end{equation*}
Thus, we easily get the second condition in \eqref{Gint} for $\g=0$ and \eqref{0int} for $\g>0$.
\\
\noindent On the other hand, for $0<k<1$, we can write
\begin{equation}\label{G1}
\ds \sqrt{1+\e^{2(k-1)}(\left|\n_{y'} g(y')\right|^2)_{|y'=x'/ \ep}}=\ep^{k-1}\sqrt{\e^{2(1-k)}+(\left|\n_{y'} g(y')\right|^2)_{|y'=x'/ \ep}}
\end{equation}  
where
\begin{equation}\label{G2}
\ds \sqrt{\e^{2(1-k)}+(\left|\n_{y'} g(y')\right|^2)_{|y'=x'/ \ep}}\, \to\, \mathcal{M}_{Y'}\left(\left|\n_{y'} g(y')\right|\right) \quad\text{ weakly * in }L^\infty(Y').
\end{equation}  
By combining \eqref{G1} and \eqref{G2} we get the third condition in \eqref{Gint} for $\g=1-k $ and \eqref{0int} for $\g>1-k$.
\end{proof}

\section{Existence and uniqueness result}\label{secexuniq}
Here, we aim to prove the existence and uniqueness of the solution of problem \eqref{probvar},  for any fixed $\ep$, together with some uniform a priori estimates. 

Let us start by proving that all the terms in \eqref{probvar} make sense.
\begin{proposition}\label{sense}
Under assumptions (\textbf{A$_g$}), $(\textbf{A}2)$ and $(\textbf{A}3)$, suppose that $r$ and $\bar{r}$ are such that
\begin{equation}
\left\{
\begin{array}{ll}
\ds i)\ \text{ if } N=2,\quad r,\,\bar{r} \in [1,+\infty),\\[2mm]
\ds ii)\ \text{ if } N>2,\quad r=\dfrac{2(N-1)}{N-2},\,\bar{r}=2^*=\dfrac{2N}{N-2},
\end{array}
\right.
\end{equation}
and 
\begin{equation}\label{cond1}
\ds r\geq q_2\,r',\qquad\bar{r}\geq q_1\bar{r}',
\end{equation}
where $r'$ and $\bar{r}'$ are the coniugate exponents of $r$ and $\bar{r}$, respectively.

Then, for any fixed $\ep$,  the following maps are bounded and continuous:
\begin{equation}\label{map}
\begin{array}{ll}
\ds i)\ H_1:w_\ep\in W_0^\ep\rightarrow H_1(w_\ep)\in L^{\frac{\bar{r}}{q_1}}(Q),
\\[2mm]
\ds ii)\ H_2:w_\ep\in W_0^\ep\rightarrow H_2(w_\ep)\in L^{\frac{r}{q_2}}(\gae),
\end{array}
\end{equation}
where $H_1(w_\ep)(x)=h_1(w_\ep(x))$, $H_2(w_\ep)(x)=h_2(w_\ep^+(x)-w_\ep^-(x))$.

Moreover for every $w_\ep$ and $v_\ep$ in $W_0^\ep$, we get
\begin{equation}\label{L1}
\left\{
\begin{array}{l}
\ds h_1(w_\ep)\,v_\ep\in L^1(Q),\\[2mm]
\ds h_2(w^+_\ep-w^-_\ep)(v^+_\ep-v^-_\ep)\in L^1(\gae).
%w^+h_i(u^+,v^+)\in L^1(O^+),\, i=1,2\\
%\\
%w^-h_i(u^-,v^-)\in L^1(O^-)\, i=1,2.
\end{array}
\right.
\end{equation}
\end{proposition}
\begin{proof}
Let us consider the lower order term. If $N>2$, the embeddings $H^1(Q_\ep^{\pm})\subset L^{2^*}(Q_\ep^{\pm})$  are continuous. Hence, we get $w_\ep\in L^{2^*}(Q)$. By $(\textbf{A}2)_{1,3}$ and the property of the Nemytskii operator associated to the function $h_1$ we get that $H_1(w_\ep)(x)=h_1(w_\ep(x))$ is a bounded and continuous map from $W_0^\ep$ to $L^{\frac{2^*}{q_1}}(Q)$. By \eqref{assqi} the first relation in \eqref{L1} holds for $N>2$.
For $N=2$ we argue in a similar way as above by using the continuous embeddings $H^1(Q_\ep^{\pm})\subset L^{\bar{r}}(Q_\ep^{\pm})$  for any $\bar{r}\in[1,+\infty)$ and again \eqref{assqi} for $N=2$. 

For the interface term, let 
$$ j: w_\ep\in W^\ep_0\,\to\, w_\ep^+-w_\ep^-\in L^r(\gae) $$
and
$$ j_2: \overline{w}_\ep\in L^r(\gae)\,\to\, j_2(\overline{w}_\ep)\in L^{\frac{r}{q_2}}(\gae) $$
be defined by $j_2(\overline{w}_\ep)(x)=h_2(\overline{w}_\ep(x))$. By the linear continuity of the inclusions $H^1(Q_\ep^+)\subset L^r(\gae)$ and  $H^1(Q_\ep^-)\subset L^r(\gae)$, $j$ is bounded and continuous. Moreover, assumption $(\textbf{A}3)_{1,4}$ and the property of the Nemytskii operator associated to the function $h_2$ imply that $j_2$ is bounded and continuous from $L^r(\gae)$ to $L^{\frac{r}{q_2}}(\gae)$. Hence, $H_2\left(w_\ep\right)=\left(j_2\circ j\right)(w_\ep)$ is a bounded and continuous map from $W^\ep_0$ into $L^{\frac{r}{q_2}}(\gae)$. Thus, by \eqref{asspi} and \eqref{cond1} the second relation in \eqref{L1} is true.
 \end{proof}
 
 At this point, we prove the main result of this section.
\begin{theorem}\label{exun}
Under assumptions $(\textbf{A}_g)$, $(\textbf{A}1)$, $(\textbf{A}2)$, $(\textbf{A}3)$ and given  $f\in L^2(Q)$,  for any fixed $\ep$, there exists a unique solution $u_\ep\in W^\ep_0$ of  problem \eqref{probvar} satisfying
\begin{equation}\label{estsol}
\left\{
\begin{array}{ll}
\ds i)\ \|\nabla u_\ep\|_{L^2(Q\setminus \gae)}\leq C,\\[2mm]
\ds ii)\ \|u^+_\ep-u^-_\ep\|_{L^2(\gae)}\leq C\ep^{-\frac{\gamma}{2}},\\[2mm]
\end{array}
\right.
\end{equation}
with $C$ positive constant independent of $\ep$.
\end{theorem}
\begin{proof} The proof follows the same arguments of Theorem 3.1 in \cite{donngu}.  We start proving, for any fixed $\ep$, the existence of a solution of the variational problem \eqref{probvar} by using the Minty-Browder theorem for the operator $\mathcal{G}:u_\ep\in W^\ep_0\to \mathcal{G}(u_\ep)\in(W^\ep_0)'$ where
\begin{equation}\label{funG}
	\ds \langle\mathcal{G}\left(u_\ep\right),\varphi_\ep\rangle_{(W^\ep_0)',W^\ep_0}=
	{\displaystyle \int_{Q \setminus \gae} D^\ep \n u_\ep \n \varphi_\ep\, dx}+ {\displaystyle \int_{\gae}h_2(u_\ep^+- u_\ep^-)(\varphi_\ep^+-\varphi_\ep^-)\, d\sigma}+ {\displaystyle \int_{Q} h_1(u_\ep) \varphi_\ep\, dx}
	-{\displaystyle \int_{Q} f \varphi_\ep\, dx}.
	\end{equation}
The continuity of $\mathcal{G}$ follows directly from assumption $(\textbf{A}1)$ and Proposition \ref{sense}. Moreover hypotheses $(\textbf{A}1)$, $(\textbf{A}2)_2$ and $(\textbf{A}3)_2$,  assure the monotonicity of $\mathcal{G}$. It remains to prove the coercivness of $\mathcal{G}$. To this aim let us observe that  by $(\textbf{A}1)$, $(\textbf{A}3)_3$, Remark \ref{remfi}, \eqref{normwe0} and \eqref{P}, it holds 
\begin{equation}\label{coeG}
\begin{array}{c}
\ds\langle\mathcal{G}\left(u_\ep\right),u_\ep\rangle\geq \alpha\left\|\n u_\ep\right\|^2_{L^2(Q\setminus \gae)}+C\left\|u_\ep^+-u_\ep^-\right\|^2_{L^2(\gae)}-\|f\|_{L^2(Q)}\|u_\ep\|_{L^2(Q)}\\[2mm]
\ds \geq C_1\|u_\ep\|^2_{W^\ep_0}-C_2\|u_\ep\|_{W^\ep_0},
\end{array}
\end{equation}
with $C_1$ and $C_2$ positive constants. The above inequality yields that the operator $\mathcal{G}$ is coercive. Hence, by the Minty-Browder theorem, $\mathcal{G}$ is surjective, so that there exists at least one function $u_\ep\in W^\ep_0$ such that $\mathcal{G}(u_\ep)=0$. This means that $u_\ep$ is a solution of problem \eqref{probvar}.
As for the uniqueness, let us suppose that problem  \eqref{probvar} admits another solution $\bar{u}_\ep$. Choosing as test function $\varphi_\ep=u_\ep-\bar{u}_\ep$ in \eqref{probvar} written in correspondence of the solutions $u_\ep$ and $\bar{u}_\ep$, respectively, we get
\begin{equation*} 
	\begin{array}{c}
	{\displaystyle \int_{Q \setminus \gae} D^\ep \n u_\ep \n (u_\ep-\bar{u}_\ep)\, dx}+
	{\displaystyle \int_{\gae}h_2(u_\ep^+- u_\ep^-)((u_\ep^+-\bar{u}_\ep^+)-(u_\ep^--\bar{u}_\ep^-))\, d\sigma}\\[5mm] 
{\displaystyle +\int_{Q} h_1(u_\ep) (u_\ep-\bar{u}_\ep)\, dx}={\displaystyle \int_{Q} f (u_\ep-\bar{u}_\ep)\, dx}.
	\end{array}
	\end{equation*}
	and
\begin{equation*} 
	\begin{array}{c}
	{\displaystyle \int_{Q \setminus \gae} D^\ep \n \bar{u}_\ep\n (u_\ep-\bar{u}_\ep)\, dx}
	{\displaystyle \int_{\gae}h_2(\bar{u}_\ep^+- \bar{u}_\ep^-)((u_\ep^+-\bar{u}_\ep^+)-(u_\ep^--\bar{u}_\ep^-))\, d\sigma}\\[5mm] 
  {\displaystyle +\int_{Q} h_1(\bar{u}_\ep) (u_\ep-\bar{u}_\ep)\, dx}= {\displaystyle \int_{Q} f (u_\ep-\bar{u}_\ep)\, dx}
  	\end{array}
	\end{equation*}
Then, we subtract and obtain
\begin{equation}\label{uniq} 
	\begin{array}{l}
	{\displaystyle \int_{Q \setminus \gae} D^\ep (\n u_\ep-\n \bar{u}_\ep) (\n u_\ep-\n \bar{u}_\ep)\, dx}\\[5mm]
	{\displaystyle +\int_{\gae}\left[h_2(u_\ep^+- u_\ep^-)-h_2(\bar{u}_\ep^+- \bar{u}_\ep^-)\right]((u_\ep^+-\bar{u}_\ep^+)-(u_\ep^--\bar{u}_\ep^-))\, d\sigma}\\[4mm] 
+{\displaystyle \int_{Q}\left[ h_1(u_\ep)-h_1(\bar{u}_\ep)\right] (u_\ep-\bar{u}_\ep)\, dx}=0.
	\end{array}
	\end{equation}
By assumption $(\textbf{A}1)$, the monotonicity of $h_1$ and $h_2$, equality \eqref{uniq} gives
\begin{equation*}
%\begin{array}{ll}
\ds \n (u_\ep-\bar{u}_\ep)=0 \quad \text{ in }Q\setminus \gae.
%\\
%u_\ep^+-u_\ep^-=\bar{u}^+-\bar{u}^-&\text{ on }\gae.
%\end{array}
\end{equation*}
Then, by using Poincar\'{e} inequality \eqref{P} we get $u_\ep=\bar{u}_\ep$ which ends the proof of uniqueness.

In order to prove uniform estimates in \eqref{estsol}, let us take $u_\ep$ as test function in \eqref{probvar} and obtain
\begin{equation*}
	{\displaystyle \int_{Q \setminus \gae} D^\e \n u_\e \n u_\e\, dx}+ \e^{\gamma} {\displaystyle \int_{\gae}h_2(u^+_\e - u^-_\e)(u_\e^+-u_\e^-)\, d\sigma}+ {\displaystyle \int_{Q} h_1(u_\e) u_\e\, dx}={\displaystyle \int_{Q} fu_\e\, dx}.
\end{equation*}
By $(\textbf{A}1)$, $(\textbf{A}3)_3$ and Remark \ref{remfi} we get
\begin{equation*}
\begin{array}{c}
\ds \alpha\|\n u_\ep\|^2_{L^2(Q\setminus \gae)}+C\,\e^{\gamma}\|u^+_\e - u^-_\e\|^2_{L^2(\gae)}\leq \|f\|_{L^2(Q)}\|u_\ep\|_{L^2(Q)},
\end{array}
\end{equation*}
which implies the claimed result in view of \eqref{P}.
\end{proof}

\section{Homogenization}\label{sechomo}
Let us start with some preliminary convergence results for the sequence of solutions of problem \eqref{probvar} needed for the proof of our main result.
They follow by \eqref{normwe0}, Proposition \ref{propdonpiat1} and the a priori estimate \eqref{estsol}$_{i}$ in Theorem \ref{exun}.
\begin{corollary}\label{corcomp}
Let $u_\ep$ be a solution of problem \eqref{probvar}. Then, there exist a subsequence (still denoted $\ep$) and a function $u\in W^0_0$ such that
\begin{equation*}
		\left\{\begin{array}{lll}
		\ds u_{\e}\rw u &\text{ strongly in } L^2(Q) \text{ and a.e. in }Q,\\[2mm]
		\ds \chi\strut_{Q^+_\e}\n u_{\e} \ru \chi\strut_{Q^+} \n u &\text{ weakly in }(L^2(Q))^N,\\[2mm]
		\ds \chi\strut_{Q^-_\e}\n u_{\e} \ru \chi\strut_{Q^-} \n u &\text{ weakly in }(L^2(Q))^N.
		\end{array}\right.
		\end{equation*}
\end{corollary}
\begin{remark}\label{rempro}
Since $\chi_{Q^+_\ep}$ and $\chi_{Q^+_\ep}$ converge a.e. to $\chi_{Q^+}$ and $\chi_{Q^-}$, respectively, as an immediate consequence of Corollary \ref{corcomp} one has 
\begin{equation*}
\left\{
\begin{array}{ll}
\ds \chi_{Q^+_\ep}u_\ep\rw \chi_{Q^+}u & \text{ strongly in }L^2(Q),\\[2mm]
\ds \chi_{Q^-_\ep}u_\ep\rw \chi_{Q^-}u & \text{ strongly in }L^2(Q).
\end{array}
\right.
\end{equation*}
 \end{remark}
By arguing as in Proposition 3.1 of \cite{donpiat}, we get the convergence result below.
\begin{corollary}\label{prophomou} Let $k>0$ and $u_\ep$ be the solution of problem \eqref{probvar}. Then, there exists a subsequence (still denoted $\ep$) such that
\begin{equation}%\label{limD}
\left\{
\begin{array}{lll}
\ds i)\ \chi_{Q^+_\ep} D^\ep \n u_\ep \ru \chi_{Q^+}D^0\n u & \text{ weakly in }(L^2(Q))^N,\\[2mm]
\ds ii)\ \chi_{Q^-_\ep} D^\ep \n u_\ep \ru \chi_{Q^-}D^0\n u & \text{ weakly in }(L^2(Q))^N.
\end{array}
\right.
\end{equation}
\end{corollary}

\subsection{Proof of Theorem \ref{theo:homores2}}
 At first we consider cases  A) and  B). To this aim, observe that given $\varphi\in W^0_0$ there exist $\psi_1$ and $\psi_2$ in $H^1_0(Q)$ such that $\varphi^+=\psi_{1|Q^+}$ and $\varphi^-=\psi_{2|Q^-}$. 
\\
Set $\varphi_\ep=\chi_{Q^+_\ep}\psi_{1}+ \chi_{Q^-_\ep}\psi_{2}$, one has that $\varphi_\ep \in W^\ep_0$. By construction and using \eqref{convchi}, it holds
	\begin{equation}\label{convtest}
		\ds \varphi_{\e}\rw \varphi \quad \text{ strongly in } L^{2^*}(Q).
		\end{equation}
Now, let us take  $\varphi_\ep$ as test function in the variational equality \eqref{probvar}. We get
\begin{equation}\label{homovar}
	{\displaystyle \int_{Q \setminus \gae} D^\e \n u_\e \n \varphi_{\ep}\, dx}+ \e^{\gamma} {\displaystyle \int_{\gae}h_2(u^+_\e - u^-_\e)(\varphi_{\ep}^+-\varphi_{\ep}^-)\, d\sigma}+{\displaystyle \int_{Q} h_1(u_\e) \varphi_{\ep}\, dx}=
{\displaystyle \int_{Q} f\varphi_{\ep}\, dx}.
	\end{equation}
We would like to pass to the limit in \eqref{homovar}. 

By Corollary \ref{prophomou}, we have, up to a subsequence, 
\begin{equation}\label{k=1lim1}
\begin{array}{c}
\displaystyle \int_{Q \setminus \gae} D^\e \n u_\e \n \varphi_{\ep}\, dx=\displaystyle \int_{Q} \chi_{Q^+_\ep}D^\e \n u_\e \n\psi_{1}\, dx+\int_{Q} \chi_{Q^-_\ep}D^\e \n u_\e \n\psi_{2}\, dx \\[5mm]
\displaystyle\rightarrow \int_{Q} \chi_{Q^+}D^0 \n u\n\psi_{1}\, dx +\int_{Q} \chi_{Q^-}D^0 \n u \n\psi_{2}\, dx=\int_{Q\setminus\Gamma_0}D^0\n u\n \varphi\,dx.
\end{array}
\end{equation}

Moreover, \eqref{convtest} gives
\begin{equation}\label{k=1lim3}
\ds \int_{Q} f \varphi_\ep\, dx \ \rw\ \int_{Q} f \varphi\, dx.
\end{equation}

In order to study the other two terms, at first, let us consider more regular functions $\psi_1$ and $\psi_2$ in $\mathcal{D}(\Omega)$. Hence, \eqref{convtest} becomes
\begin{equation}\label{convtestbis}
		\ds \varphi_{\e}\rw \varphi\quad \text{ strongly in } L^p(Q), \ \text{ for any }p\geq 1.
		\end{equation}

Concerning the lower order term, by $(\textbf{A}2)_{1,3}$, Corollary \ref{corcomp},  the properties of the Nemytskii operator associated to function $h_1$ and \eqref{convtestbis} written for $p=\left(\frac{2}{q_1}\right)'$, we get,  up to a subsequence, \bk
\begin{equation}\label{k=1lim2}
\ds \int_{Q} h_1(u_\ep) \varphi_{\ep}\, dx\ \rightarrow\ \int_{Q}h_1(u)\varphi\,dx.
\end{equation}

Now, let us treat the interface integral. By Remark  \ref{reminterface}, we have that 
\begin{equation}\label{sur1}
\begin{array}{c}
\ds \ep^\g{\displaystyle \int_{\gae}h_2(u^+_\e-u^-_\e)(\varphi_{\ep}^+-\varphi_{\ep}^-)\, d\sigma}\\[5mm]
{\displaystyle =\ep^{\g} \int_{\omega}h_2\left(u^+_\e\left(x',\ep g\left(\frac{x'}{\ep}\right)\right) - u^-_\e\left(x',\ep g\left(\frac{x'}{\ep}\right)\right)\right)}\times\\[5mm]
 \displaystyle{\left(\psi_{1}\left(x',\ep g\left(\frac{x'}{\ep}\right)\right)-\psi_{2}\left(x',\ep g\left(\frac{x'}{\ep}\right)\right)\right)}\times {\displaystyle\sqrt{1+\ep^{2(k-1)}(\left|\n_{y'} g(y')\right|^2)_{|y'=x'/ \ep}}\, dx'.}
\end{array}
\end{equation}
In view of \eqref{estsol}$_i$, by Proposition \ref{propdonpiat2} one has,  up to a subsequence, \bk
\begin{equation}\label{convtrace}
\ds u^{\pm}_\e\left(\cdot,\ep g\left(\frac{\cdot}{\ep}\right)\right)\ \rw\ u^{\pm}(\cdot,0)\, \text{ strongly in } L^2(\omega).
\end{equation}
Thus, by $({\bf A}3)_{1,4}$, \eqref{convtrace} and again the properties of the Nemytskii operator associated to $h_2$, we get
\begin{equation}\label{convfh}
\ds h_2\left( u^+_\e\left(\cdot,\ep g\left(\frac{\cdot}{\ep}\right)\right) -  u^-_\e\left(\cdot,\ep g\left(\frac{\cdot}{\ep}\right)\right)\right)\ \rw\ h_2\left(u^+(\cdot,0)-u^-(\cdot,0)\right) \text{ strongly in } L^{\frac{2}{q_2}}(\omega).
\end{equation}
On the other hand, in view of Lemma \ref{lemCeFrPiat} and since $\psi_i\in\mathcal{D}(Q)$, for every $p\geq 1$  we get
\begin{equation}\label{estpsi1}
\ds \left\|\psi_i\left(x',\ep g\left(\frac{x'}{\ep}\right)\right)-\psi_i\left(x',0\right)\right\|_{L^p(\omega)}\leq C\ep^{\frac{k}{p}}\|\psi_i\|_{W^{1,p}(Q)},
\end{equation}
i.e. 
\begin{equation}\label{estpsi1'}
\ds\psi_i\left(x',\ep g\left(\frac{x'}{\ep}\right)\right)\ \to\ \psi_i\left(x',0\right)\quad \text{ strongly in  } L^p(\omega), 
\end{equation}
for any fixed $i\in \{1,2\}$.

Hence, we are able to pass to the limit  in \eqref{sur1} obtaining different convergence results according to the values of $k$ and $\g$. 

By Proposition \ref{propg},  \eqref{sur1}, \eqref{convfh} and \eqref{estpsi1'} written for $p=\left(\frac{2}{q_2}\right)'$, in case A) we have,  up to a subsequence, \bk 
 \begin{equation}\label{k<1lim6'}
{\displaystyle \ep^{\gamma} \int_{\gae}h_2(u^+_\e-u^-_\e)(\varphi_{\ep}^+-\varphi_{\ep}^-)\, d\sigma }\\[2mm]
\rw 
\displaystyle  G \; \int_{\omega}h_2(u^+\left(x',0\right)-u^-\left(x',0\right))(\psi_1\left(x',0\right)-\psi_2\left(x',0\right))\, d x', 
\end{equation}
for any $\psi_1$ and $\psi_2$ in $\mathcal{D}(Q)$, where $G$ is defined in \eqref{Gint}, while in case B) it holds
\begin{equation}\label{k<1lim6''}
{\displaystyle \ep^{\gamma} \int_{\gae}h_2(u^+_\e-u^-_\e)(\varphi_{\ep}^+-\varphi_{\ep}^-)\, d\sigma }\ \rw\ 0.
\end{equation}

Now, in both cases, we fix $i\in \{1,2\}$ and take $\psi_i\in H_0^1(Q)$. Then, there exists a sequence $\{{\psi_{i}}_n \} \subset \mathcal{D}(Q)$ such that 
\begin{equation*}
\ds {\psi_{i}}_n\rightarrow \psi_i\quad \text{ strongly in }H_0^1(Q),
\end{equation*}
which implies
\begin{equation}\label{density}
\ds {\psi_{i}}_n\rightarrow \psi_i\quad \text{ strongly in }L^p(Q),\text{ for any }p\in[1,2^*].
\end{equation}
Set $\varphi_n=\chi_{Q^+}{\psi_{1}}_n+\chi_{Q^-}{\psi_{2}}_n$, it holds
\begin{equation}\label{conf}
\ds \varphi_n\rightarrow \varphi\quad \text{ strongly in }L^{2^*}(Q).
\end{equation}
We observe that if $N>2$ the embeddings $H^1(Q^{\pm})\subset L^{2^*}(Q^{\pm})$  are continuous, which give $u\in L^{2^*}(Q)$. Therefore, by $(\textbf{A$2$})_{3}$ and \eqref{conf} we obtain 
\begin{equation}\label{k=1lim4}
\ds\int_{Q}h_1(u)\varphi_{n}\,dx\ \rightarrow\ \int_{Q}h_1(u)\varphi \,dx.
\end{equation}
For $N=2$ we argue as above using the continuous embeddings $H^1(Q^{\pm})\subset L^{\bar{r}}(Q^{\pm})$  for any $\bar{r}\in[1,+\infty)$ and again $(\textbf{A$2$})_{3}$. 

Regarding the interface term
$$
G\int_{\omega} h_2\left(u^+(x',0) - u^-(x',0)\right)({\psi_{1}}_n(x',0)-{\psi_{2}}_n(x',0))\, d x',
$$
the linear continuity of the inclusions $H^1(Q^{\pm})\subset L^r(\Gamma_0)$, where
\begin{equation*}
\left\{\begin{array}{ll}
\ds\hbox{ if }N=2, \quad  r\in [1,+\infty),\\[2mm]
\ds\hbox{ if } N>2,  \quad r=\dfrac{2(N-1)}{N-2},
\end{array}
\right.
\end{equation*}
yields that, for $i\in \{1,2\}$, ${\psi_{i}}_n|_{\Gamma_{0}}$ is in $L^r(\Gamma_0)$. Thus, ${\psi_{i}}_n(\cdot,0)$ belongs to $L^r(\omega)$.
\\
By $(\textbf{A$3$})_4$ we always have $r'\leq \dfrac{r}{q_2}$ since $\dfrac{r}{r'}=\dfrac{N}{N-2}$. Hence, we get
$$
\begin{array}{c}
 \displaystyle G\int_{\omega} h_2\left(u^+(x',0) - u^-(x',0)\right)({\psi_{1}}_n(x',0)-{\psi_{2}}_n(x',0))\, d x'\\[5mm]
\displaystyle \rightarrow\ G\int_{\omega} h_2\left(u^+(x',0) - u^-(x',0)\right)(\psi_{1}(x',0)-\psi_{2}(x',0))\, d x',
\end{array}
$$
for $\psi_1$ and $\psi_2$ in $H^1_0(Q)$ where $G$ is defined in \eqref{Gint} of Proposition \ref{propg}.

The above results allow us to have, by a density argument, the limit problems \eqref{prolimi1} and \eqref{prolimi2}.

\vskip 3 mm

Now, it remains to treat case C). To this aim, let us prove that 
\begin{equation}\label{traccia}
u^+_{|\Gamma_0}=u^-_{|\Gamma_0} =0.
\end{equation} 

\smallskip

For $k\geq 1$ and $\g<0$, by Remark \ref{reminterface}, Proposition \ref{propdonpiat2} and  \eqref{estsol}, we have,  up to a subsequence, \bk
\begin{equation}\label{gamma<0}
\begin{array}{c}
{\displaystyle \int_{\gae}(u^+_\e - u^-_\e)^2\, d\sigma}\\[5mm]
{\displaystyle = \int_{\omega}\left(u^+_\e\left(x',\ep g\left(\frac{x'}{\ep}\right)\right) - u^-_\e\left(x',\ep g\left(\frac{x'}{\ep}\right)\right)\right)^2}\times {\displaystyle\sqrt{1+\e^{2(k-1)}(\left|\n_{y'} g(y')\right|^2)_{|y'=x'/ \ep}}\, dx'}\\[5mm]
\ds \rightarrow 
\left\{\begin{array}{ll}
{\displaystyle \int_{\omega} \left(u^+(x',0) - u^-(x',0)\right)^2\, d x'} & \hbox{ if }k>1, \\[3mm]
{\displaystyle \mathcal{M}_{Y'}\left(\sqrt{1+(\left|\n_{y'} g(y')\right|^2)}\right) \int_{\omega}  \left(u^+(x',0) - u^-(x',0)\right)^2\, d x'} & \hbox{ if }k=1
		\end{array}\right.
		\end{array}
		\end{equation} 
and
\begin{equation}\label{gamma<0'}
{\displaystyle \int_{\gae}(u^+_\e - u^-_\e)^2\, d\sigma}\leq \e^{-\gamma}\ \rightarrow\\ 0.
	\end{equation} 
\smallskip
Thus, \eqref{traccia} holds by uniqueness.

For $0<k<1$ and $\g<1-k$, again by Remark \ref{reminterface}, Proposition \ref{propdonpiat2} and  \eqref{estsol}, one obtains, up to a subsequence, \bk
\begin{equation}\label{gamma<0*}
\begin{array}{c}
{\displaystyle  \ep^{1-k}\int_{\gae}(u^+_\e - u^-_\e)^2\, d\sigma}=\\[5mm]
{\displaystyle \int_{\omega}\left(u^+_\e\left(x',\ep g\left(\frac{x'}{\ep}\right)\right) - u^-_\e\left(x',\ep g\left(\frac{x'}{\ep}\right)\right)\right)^2}\times {\displaystyle\sqrt{\e^{2(1-k)}+(\left|\n_{y'} g(y')\right|^2)_{|y'=x'/ \ep}}\, dx'}\\[5mm]
\ds \rightarrow \mathcal{M}_{Y'}\left(\left|\n_{y'} g(y')\right|\right)\displaystyle \int_{\omega} \left(u^+(x',0) - u^-(x',0)\right)^2\, d x',
\end{array}
\end{equation}
and
\begin{equation}\label{gamma<0**}
{\displaystyle  \ep^{1-k}\int_{\gae}(u^+_\e - u^-_\e)^2\, d\sigma}\leq \e^{1-k-\g}\ \rightarrow\ 0.
\end{equation}
Hence, again we get \eqref{traccia} by uniqueness. 

Thus, we can choose a test function $\varphi\in H_0^1(Q)$  in \eqref{homovar}. As a consequence, the interface term disappears and one easily obtains the limit problem \eqref{prolimi3}.

\medskip

In conclusion, let us observe that since the limit problems have a unique solution,  the convergences in \eqref{convfin} hold for the whole sequence. \hfill $\Box$

\end{document}